\documentclass{amsart}

\usepackage{amsmath}
\usepackage{amsbsy}
\usepackage{amssymb}
\usepackage{amsthm}
\usepackage{amscd}
\usepackage{mathrsfs}
\usepackage[abbrev]{amsrefs}
\usepackage[all]{xy}
\newcommand{\bba}{{\mathbb A}}

\newcommand{\bbf}{{\mathbb F}}

\newcommand{\bbq}{{\mathbb Q}}

\newcommand{\bbz}{{\mathbb Z}}

\newcommand{\gO}{{\mathfrak O}}

\font\tenscr=rsfs10 

\newcommand{\sC}{\hbox{\tenscr C}}

\newcommand{\gal}{{\operatorname{Gal}}}

\newcommand{\Llrarrow}{\Longleftrightarrow}

\newcommand{\lrarrow}{\longrightarrow}

\newcommand{\kableadd}%
{Department of Mathematics\\ Cornell University\\
Ithaca NY 14853}

\newcommand{\lan}{\langle}
\newcommand{\ran}{\rangle}

\newtheorem{thm}{Theorem}[section]
\newtheorem{lem}[thm]{Lemma}

\newtheorem{cond}[thm]{Condition}
\newtheorem{rem}[thm]{Remark}        
\newtheorem*{ack*}{Acknowledgment}

\newtheorem*{note*}{Notation and conventions}

\DeclareFontEncoding{OT2}{}{}
\DeclareFontSubstitution{OT2}{cmr}{m}{n}

\DeclareFontFamily{OT2}{cmr}{\hyphenchar\font45 }
\DeclareFontShape{OT2}{cmr}{m}{n}{%
   <5><6><7><8><9>gen*wncyr%
   <10><10.95><12><14.4><17.28><20.74><24.88>wncyr10}{}
\DeclareFontShape{OT2}{cmr}{b}{n}{%
   <5><6><7><8><9>gen*wncyb%
   <10><10.95><12><14.4><17.28><20.74><24.88>wncyb10}{}

\DeclareMathAlphabet{\mathcyr}{OT2}{cmr}{m}{n}
\DeclareMathAlphabet{\mathcyb}{OT2}{cmr}{b}{n}
\SetMathAlphabet{\mathcyr}{bold}{OT2}{cmr}{b}{n}

\begin{document}
\pagestyle{headings}

\title{Rational Points on Diagonal Cubic Surfaces}
\author{Kazuki Sato}
\address{Mathematical Institute, Tohoku University, Sendai, Miyagi, 980-8578, Japan}
\email{sb0m17@math.tohoku.ac.jp}
\begin{abstract}
We show under the assumption that the Tate-Shafarevich group of any elliptic curve over $\bbq$ is finite that the cubic surface $x_1^3 + p_1p_2x_2^3 + p_2p_3x_3^3 + p_3p_1x_4^3 = 0$ over $\bbq$ has a rational point, where $p_1, p_2$ and $p_3$ are rational primes congruent to $2$ or $5$ modulo $9$.
\end{abstract}
\subjclass[2010]{11D25,14G05}
\keywords{cubic surface, rational point, Selmer group}
\thanks{}
\date{\today}
\maketitle

\renewcommand\baselinestretch{1.1}

\setcounter{section}{-1}
\section{Introduction}
Let $V$ be the smooth projective cubic surface over $\bbq$ defined by
\[ a_1x_1^3 + a_2x_2^3 + a_3x_3^3 + a_4x_4^3 = 0,\]
where $a_1, a_2, a_3, a_4$ are non-zero integers.
It is natural to ask when $V$ has a $\bbq$-rational point.
If $V$ has a $\bbq$-rational point, then obviously $V$ has a $\bbq_p$-rational point for every prime $p$.
We say that the Hasse principle holds if the converse is also true.
Selmer \cite{sel2} showed that if one of the ratios $a_1a_2/a_3a_4$, $a_1a_3/a_2a_4$, $a_1a_4/a_2a_3$ is in $(\bbq^*)^3$ then the Hasse principle holds for $V$.
However the Hasse principle does not hold in general for diagonal cubic surfaces.
The first counterexample
\[ 5x_1^3 +9x_2^3 + 10x_3^3 + 12x_4^3 = 0 \]
was found by Cassels and Guy \cite{cg}.
After that Colliot-Th{\'e}l{\`e}ne, Kanevsky, and Sansuc \cite{ct} gave a lot of counterexamples by calculating the Brauer-Manin obstruction, and showed that if $0 < a_i < 100$ then counterexamples to the Hasse principle are explained by the Brauer-Manin obstruction.
It is conjectured that the Brauer-Manin obstruction is the only obstruction to the Hasse principle for smooth cubic surfaces (see, for example, \cite{sk} for the Brauer-Manin obstruction). 

As for the existence of $\bbq$-rational points on $V$, the remarkable papers are \cite{bf}, \cite{hb}, and \cite{sd}.
%
They are based on the following simple idea;
if we can find a non-zero integer $B$ such that each of the curves
\[a_1x_1^3 + a_2x_2^3 = Bx_0^3, \ \   a_3x_3^3 + a_4x_4^3 = Bx_0^3\]
has a $\bbq$-rational point, then so does $V$.
In order to get a $\bbq$-rational point on each of the curves, it is sufficient to show first that these have a $\bbq_p$-rational point for all $p$, and second that the Hasse principle holds.
Under the assumption that the Tate-Shafarevich group of the elliptic curve of the form
\begin{equation}
 x^3 + y^3  = Az^3 \label{el}
\end{equation}
(the Jacobian of the above curve) is finite, the above strategy works well.
In \cite{hb}, one has to assume that the Selmer conjecture holds instead of the finiteness of the Tate-Shafarevich group.

Swinnerton-Dyer \cite{sd} proved that if $V$ has a $\bbq_p$-rational point for any $p$ and the Tate-Shafarevich group of any elliptic curve (\ref{el}) over any quadratic field is finite, then certain local conditions for the coefficients of $V$ is sufficient for the existence of $\bbq$-rational points, which generalizes \cite{bf} (see \cite{sd} or section $3$ for the precise statement).
In his proof, a descent argument in \cite{bf} plays an important role.

In this paper, we prove the following theorem with the use of Lemma 1.1.
\begin{thm}[Theorem \ref{mainthm}]
Let $p_1, p_2, p_3$ be rational primes such that $p_1,p_2, p_3 \equiv 2, 5 \bmod 9$, and 
\[ V : x_1^3 + p_1p_2x_2^3 + p_2p_3x_3^3 + p_3p_1x_4^3 = 0 \]
 a cubic surface over $\bbq$.
Assume that the Tate-Shafarevich group of any elliptic curve over $\bbq$ is finite.
Then $V(\bbq) \neq \emptyset$.
\end{thm}

\begin{rem}\normalfont
i) We see easily that $V$ has a $\bbq_p$-rational point for any prime $p$.

ii) $V$ is birationally equivalent to a plane over $\bbq_{p_i}$ for $i = 1,2,3$, but not over $\bbq_3$.
Hence it follows from \cite[Proposition 2]{ct} that the Brauer-Manin obstruction to the Hasse principle is empty.

iii) We cannot apply \cite[Theorem 3]{sd} to this case (see Remark \ref{rem34} for the details).

iv) We do not need to assume the finiteness of the Tate-Shafarevich groups of all elliptic curves (Remark \ref{rem1}).
For example, it is sufficient to assume that the Tate-Shafarevich group of the elliptic curve
\[ x^3 + y^3 = p_1p_2p_3^2z^3 \]
over $\bbq$ is finite if $(p_1,p_2,p_3) \equiv (2,2,5)$ or $(5,5,2) \bmod 9$.
\end{rem}


We give an overview of this paper.
In section $1$, we recall briefly the argument on descent developed by Basile and Fisher and introduce some notations, which are used throughout this paper.
For the details of section $1$, see \cite{bf}.
In section $2$, we calculate explicitly the $\sqrt{-3}$-Selmer group $S(A)$ (see section $1$ for the definition) for a cube free integer $A \in \bbz \setminus \{ 0, \pm 1 \}$, and prove Theorem \ref{mainthm}.
In section $3$, we give some sufficient conditions for the diagonal cubic surface $V$ to have a $\bbq$-rational point (Theorem \ref{sum}),
which is an analog of \cite[Theorem 1]{sd}.
Our proof is based on \cite[Theorem 3]{sd}.
It turns out in Remark \ref{rem34} that Theorem \ref{sum} (hence \cite[Theorem 3]{sd}) does not imply Theorem \ref{mainthm}.


\section{Descent on a Elliptic Curve}
Let $\zeta_3 \in \bar{\bbq}$ be a fixed primitive cube root of unity and $k = \bbq(\zeta_3)$.
For a cube free integer $A \in \bbz \setminus \{ 0,\pm 1 \} $, $E_A$ denotes the elliptic curve defined by
\[ E_A : x^3 + y^3 = Az^3\]
with identity $O = ( 1, -1, 0)$.
$E_A$ admits complex multiplication, so that $\operatorname{End}_k(E_A) = \bbz [\zeta_3 ]$.
In fact, $\zeta_3$ acts on $E_A$ by $(x,y,z) \mapsto (x,y,\zeta_3 z)$,
and the multiplication-by-$\sqrt{-3}$ endomorphism on $E_A$ is given by
\[ \sqrt{-3} : (x,y,z) \mapsto ( \zeta_3 x^3 - \zeta_3^2 y^3, \zeta_3 y^3 - \zeta_3^2 x^3, (\zeta_3- \zeta_3^2 ) xyz ).\]
Note that the group $E_A(\bar{k})[\sqrt{-3}]$ of $\sqrt{-3}$-torsion points is isomorphic to the group $\mu_3$ of cube roots of unity as a $\gal (\bar{\bbq} / k)$-module.
By Kummer theory, we have the short exact sequence
\[ 0 \lrarrow E_A(k)/\sqrt{-3} E_A(k) \lrarrow H^1(k, E_A[ \sqrt{-3}]) \overset{f}{\lrarrow} H^1(k, E_A)[\sqrt{-3} ] \lrarrow 0.\]
With the identification $H^1(k, E_A[ \sqrt{-3}]) \cong k^*/(k^*)^3$, for $\alpha \in k^*$, the image of $\alpha(k^*)^3$ under the above map $f$ is the class of the $k$-torsor
\[ C_{A,\alpha} : \alpha x^3 + \alpha^{-1}y^3 = Az^3 \]
under $E_A$.
We denote by $S(A)$ the $\sqrt{-3}$-Slemer group $S^{(\sqrt{-3})}(E_A /k)$, which is defined by
\[ S(A) = \ker[H^1(k, E_A[ \sqrt{-3}]) \lrarrow H^1(k, E_A) \lrarrow \prod_v H^1(k_v, E_A) ],  \]
where $v$ runs over all places of $k$.
Then $S(A)$ consists of $\alpha(k^*)^3 \in k^*/(k^*)^3$ for which the corresponding curve $C_{A,\alpha}$ has a $k_v$-rational point for all $v$.
We denote by $C(A)$ the subgroup of $S(A)$ consisting of elements which lies in the kernel of $f$.
By definition, $\alpha(k^*)^3 \in C(A)$ corresponds to the curve $C_{A,\alpha}$ having a $k$-rational point, and
we have the exact sequence
\[ 0 \lrarrow C(A) \lrarrow S(A) \lrarrow \mathcyr{Sh}(E_A/k)[\sqrt{-3} ] \lrarrow 0,\]
where $\mathcyr{Sh}(E_A/k)$ is the Tate-Shafarevich group of $E_A$ over $k$.
We note that $C(A)$ contains a nontrivial element $A(k^*)^3$ since $C_{A,A}$ has a point $(1,0,1)$.
Then Basile and Fisher showed the following lemma.

\begin{lem}[{\cite{bf}}]\label{bf}
%
%
%
Notations are the same as above.
Assume that the Tate-Shafarevich group $\mathcyr{Sh}(E_A/\bbq)$ of $E_A$ over $\bbq$ is finite.
If the order of $S(A)$ is $9$, then $C(A) = S(A)$.
\end{lem}

\section{Calculation of Selmer groups}
Let $A \in \bbz \setminus \{ 0, \pm 1 \}$ be a cube free integer.
In this section we calculate the Selmer group $S(A)$ for some $A$.
According to the previous section, we regard $S(A)$ as a subgroup of $k^*/(k^*)^3$.
For $\alpha \in k^*$ our task is to determine the condition for $C_{A,\alpha}$ to have a $k_v$-rational point for all places $v$.
In the following, we identify an element of $k^*$ with its image in $k^*/(k^*)^3$.
For a prime element $q$ in $\bbz[\zeta_3]$, denote by $k_q$ the completion of $k$ for the topology defined by the $(q)$-adic valuation. 

\begin{lem}\label{3.1}
Let $\alpha \in \bbz[\zeta_3]$ be a non-zero cube free integer.
If $q|\alpha$ is a prime in $\bbz[\zeta_3]$ such that $q\not|A$,
then $C_{A,\alpha}(k_q) = \emptyset$.
\end{lem}
\begin{proof}
Assume that $C_{A,\alpha}(k_q) \neq \emptyset$.
For a rational point $(x,y,z) \in C_{A,\alpha}(k_q)$ we have
\[ \alpha^2 x^3 + y^3 = \alpha Az^3.\]
Let $v_q$ be the discrete valuation associated with the prime $q$.
Since $v_q(A) =0$,
\[ v_q(\alpha^2 x^3 ) \equiv 2v_q(\alpha), \ v_q(y^3) \equiv 0,\ v_q(\alpha A z^3) \equiv v_q(\alpha) \bmod 3.\]
It follows from $v_q(\alpha) \not\equiv 0 \bmod 3$ that
\[ v_q(\alpha A z^3) = \min \{ v_q(\alpha^2 x^3 ),  v_q(y^3)\}. \]
This is a contradiction.
\end{proof}

\begin{lem}\label{3.2}
Let $\alpha$ be a non-zero element in $\bbz[\zeta_3]$ and $q \nmid 3\alpha A$ a prime in $\bbz[\zeta_3]$.
Then $C_{A,\alpha}(k_q) \neq \emptyset$.
\end{lem}

\begin{proof}
The curve
\[C_{A,\alpha}: \alpha x^3 + \alpha^{-1}  y^3 = Az^3 \]
has good reduction at $q$ and let $\sC$ be the smooth projective model of $C_{A,\alpha}$ over $\gO_q$ defined by the above equation, where $\gO_q$ is the ring of integers in $k_q$.
Since the special fiber of $\sC$ is a curve of genus $1$ over a finite field, it has a rational point (see \cite{ca}).
By smoothness and Hensel's lemma, $C_{A,\alpha}$ has a $k_q$-rational point, too.
\end{proof}

\begin{lem}\label{3.3}
Let $\alpha \in \bbz[\zeta_3]$ be a non-zero cube free integer and $q$ a prime in $\bbz[\zeta_3]$ satisfying $q \| A$ and $q \nmid 3$, where 
$q^n \| A$ for an integer $n$ means that $q^n |A$ and $q^{n+1} \nmid A$.
Then
\begin{equation*}
C_{A,\alpha} (k_q) \neq \emptyset \Llrarrow 
\begin{cases}
   \left( \frac{A\alpha^{-1}}{q} \right)_3 = 1 & \textrm{if} \ q \| \alpha, \\
\\
   \left( \frac{A\alpha q^{-3}}{q} \right)_3 = 1     &  \textrm{if} \ q ^2 \| \alpha, \\
\\  
 \left( \frac{\alpha}{q} \right)_3 = 1  &  \textrm{if} \ q \nmid \alpha.
                                              \end{cases}
\end{equation*}
\end{lem}

\begin{proof}
First, we consider the case where $q \| \alpha$.
The curve
\[ C_{A,\alpha} : \alpha x^3 + \alpha^{-1}y^3 = Az^3 \]
is isomorphic to the curve defined by
\[ \frac{\alpha^2}{q^2}x^3 + qy^3 = \frac{\alpha A}{q^2}z^3. \]
This equation also defines a regular projective model $\sC$ of $C_{A, \alpha}$ over $\gO_q$.
Therefore $C_{A,\alpha}$ has a $k_q$-rational point if and only if $\sC$ has a $\gO_q$-valued point, and this is equivalent to that the special fiber of $\sC$ has a $\bbf_q$-rational point on the smooth locus, where $\bbf_q$ is the residue field of $\gO_q$.
Since the special fiber is defined by
\[ \frac{\alpha^2}{q^2}x^3  = \frac{\alpha A}{q^2}z^3, \]
we have 
\[
C_{A,\alpha}(k_q) \neq \emptyset \Llrarrow \frac{A}{\alpha} \in (\bbf_q ^*)^3
                                 \Llrarrow \left( \frac{A\alpha^{-1}}{q} \right)_3 =1.
\]

 Second, the case where $q^2 \| \alpha$.
The curve $C_{A,\alpha}$ has the regular projective model defined by
\[ q\frac{\alpha^2}{q^4}x^3 + y^3 = \frac{\alpha A}{q^3}z^3. \]
Its special fiber is defined by
\[ y^3 = \frac{\alpha A}{q^3}z^3,\]
therefore a similar argument shows that 
\[
C_{A,\alpha}(k_q) \neq \emptyset \Llrarrow \frac{\alpha A}{q^3} \in (\bbf_q ^*)^3
                                 \Llrarrow \left( \frac{A\alpha q^{-3}}{q} \right)_3 =1.
\]

Finally, $q \nmid \alpha$.
The equation 
\[ C_{A,\alpha} : \alpha x^3 + \alpha^{-1}y ^3 = Az^3 \]
gives a regular proper model, and its special fiber is
\[ \alpha x^3 + \alpha ^{-1}y^3 = 0. \]
Therefore we have
\[
C_{A,\alpha}(k_q) \neq \emptyset \Llrarrow \alpha^2 \in (\bbf_q ^*)^3 
                                 \Llrarrow \left( \frac{\alpha^2}{q} \right)_3 =1 
                                 \Llrarrow \left( \frac{\alpha}{q} \right)_3 =1.
\]
\end{proof}

\begin{lem}\label{cubicresidue}
Let $\alpha \in \bbz[\zeta_3]$ be a non-zero cube free integer and $p \equiv 2 \bmod 3$ a rational prime number satisfying
$p \| A$.
Then
\[
C_{A,\alpha} (k_p) \neq \emptyset \Llrarrow \left( \frac{\alpha p^{-v_p(\alpha)}}{p} \right) _3 =1.
\]
\end{lem}

\begin{proof}
Since $p \equiv 2 \bmod 3$, $p$ remains a prime in $\bbz[\zeta_3]$ and $\bbz_p^{*} \subset (\bbq_p^*)^3$.
Therefore if $v_p(\alpha)=1$, then
\[ \left( \frac{A\alpha^{-1}}{p} \right)_3 = \left( \frac{p\alpha^{-1}}{p} \right)_3 = \left( \frac{\alpha p^{-1}}{p} \right)_3^{-1}.\]
If $v_p(\alpha) = 2$, then
\[ \left( \frac{A\alpha p^{-3}}{p} \right)_3 = \left( \frac{\alpha p^{-2}}{p} \right)_3.\]
Hence the Lemma follows from Lemma \ref{3.3}.
\end{proof}

Let $\alpha \in k^*$ be a representative of an element in $k^*/(k^*)^3$.
Then we may assume that $\alpha \in \bbz[\zeta_3]$ is a non-zero cube free integer by multiplication by an element of $(k^*)^3$.
Let $A = \prod_{i=1}^r q_i^{n_i}$ be a prime decomposition of $A$ in $k$, where $(q_i)$ are distinct prime ideals in $\bbz[\zeta_3]$ and $n_i \geq 1$.
The condition $\alpha \in S(A)$ implies that $\alpha$ is of the form
\[ \alpha = \zeta_3^m \prod_{i=1}^{r} q_i^{m_i}, \ m, m_1,\dots, m_r \in \{ 0,1,2 \}\]
by Lemma \ref{3.1}.
By multiplication by $A \in C(A) \subset S(A)$, we may assume that 
\[ \alpha = \zeta_3^m \prod_{i=1}^{r-1} q_i^{m_i}.\]
Conversely, if the above $\alpha$ satisfies $C_{A,\alpha}(k_q) \neq \emptyset$ for every prime $q \in \bbz[\zeta_3]$ dividing $3A$, then it follows from Lemma \ref{3.2} that $\alpha \in S(A)$.

 In the following we write $\lambda = 1 - \zeta_3$.

\begin{lem}\label{lemmaa}
Let $p_1, p_2, p_3$ be distinct rational primes such that $(p_1, p_2,p_3 ) \equiv (2,2,5)$ or $(5,5,2) \bmod 9$.
If $A = p_1p_2p_3^2$, then $S(A) = \lan A, p_1p_2^2 \ran$.
\end{lem}

\begin{proof}
By lemma \ref{cubicresidue}, the curve $C_{A,\zeta_3^m p_1^{m_1}p_2^{m_2}}$ has a $k_{p_1}$-rational point if and only if 
\[ \left( \frac{\zeta_3^m p_2^{m_2}}{p_1} \right) _3 =1.  \]
Since $p_1 \equiv 2 \bmod 3$, we have
\[ \left( \frac{\zeta_3^m p_2^{m_2}}{p_1} \right) _3 = \zeta_3^{m(p_1^2-1)/3},\]
and this is equal to $1$ if and only if $m=0$.
If $m=0$, then the same argument shows that $C_{A,p_1^{m_1}p_2^{m_2}}$ has a $k_{p_2}$-rational point.
Since $p_3 \equiv 2 \bmod 3$, the curve $C_{A,p_1^{m_1}p_2^{m_2}}$ is isomorphic to
\[ x^3 + y^3 = p_3^2z^3 \]
over $k_{p_3}$.
Hence $C_{A,p_1^{m_1}p_2^{m_2}}$ has a $k_{p_3}$-rational point.

We claim that $C_{A,p_1p_2^2}(k_\lambda) \neq \emptyset$ and $C_{A,p_1}(k_\lambda) = \emptyset$.
It follows from this that $S(A) = \lan A, p_1p_2^2 \ran$.
We note that since these cubic curves are defined over $\bbq$, they have a $k_\lambda$-rational point if and only if they have a $\bbq_3$-rational point. 

The curve $C_{A,p_1p_2^2}$ is isomorphic to
\[ p_2x^3 + p_1 y^3 = p_3^2z^3 \]
over $\bbq$.
Since $p_1 \equiv p_2 \bmod 9$ we see that $p_1/p_2 \in \bbz_3^*$ is a cube in $\bbq_3$, and
the curve $C_{A,p_1p_2^2}$ has a $\bbq_3$-rational point.

The curve $C_{A,p_1}$ is isomorphic to
\[ x^3 + p_1 y^3 = p_2p_3^2z^3.\]
By modulo $9$ this equation becomes
\[ x^3 + 2 y^3 = 5z^3 \ (\textrm{resp.}\  x^3 + 5y^3 = 2z^3)\]
if $(p_1, p_2,p_3 ) \equiv (2,2,5) \bmod 9$ (resp. $(5,5,2) \bmod 9$).
It has no nontrivial solutions in $\bbz/9\bbz$.
Hence $C_{A,p_1}$ has no $\bbq_3$-rational points.
\end{proof}

\begin{lem}\label{lemmab}
Let $p_1,p_2,p_3$ be distinct rational primes such that $(p_1, p_2,p_3 ) \equiv (2,2,2)$ or $(5,5,5) \bmod 9$.
If $A = p_1^2p_2^2p_3^2$, then $S(A) = \lan A, p_1p_2^2 \ran$.
\end{lem}

\begin{proof}
Since $p_3 \equiv 2, 5 \bmod 9$, the curve $C_{A, \zeta_3 p_1^{n_1}p_2^{n_2}}$ is isomorphic to the curve
\[ C : \zeta_3 x^3 + \zeta_3^2 y^3 = p_3^2 z^3 \]
over $k_{p_3}$.
If $C$ has a $k_{p_3}$-rational point, then we see that $\zeta_3$ is a cube in $k_{p_3}$.
On the other hand, by the definition of the cubic residue symbol, we have
\[ \left( \frac{\zeta_3}{p_3} \right)_3 = \zeta_3^{\frac{p_3^2-1}{3}} \neq 1. \]
This is a contradiction, therefore we have $\zeta_3 p_1^{n_1}p_2^{n_2} \not\in S(A)$. 

The curve $C_{A,p_1p_2^2}$ is isomorphic to the curve 
\[ p_2x^3 + p_1y^3 = p_1p_2p_3^2z^3. \]
Since $\bbz_{p_i}^* \subset (\bbq_{p_i}^*)^3$, we see that $C_{A,p_1p_2^2}(k_{p_i}) \neq \emptyset$ for $i= 1,2,3$.
Moreover $C_{A,p_1p_2^2}$ has a $\bbq_3$-rational point since $p_1/p_2 \in (\bbq_3^*)^3$.
Hence we have $p_1p_2^2 \in S(A)$.

The curve $C_{A,p_1}$ is isomorphic to the curve
\[ p_1^2x^3 +  y^3 = p_2^2p_3^2z^3.\]
By modulo $9$, this equation becomes
\[ 4x^3 +  y^3 = 7z^3 \ (\textrm{resp.}\  7x^3 + y^3 = 4z^3)\]
if $(p_1, p_2,p_3 ) \equiv (2,2,2) \bmod 9$ (resp. $(5,5,5) \bmod 9$).
It has no nontrivial solutions in $\bbz/9\bbz$.
Therefore $C_{A,p_1}$ has no $\bbq_3$-rational points.
This completes the proof.
\end{proof}

\begin{rem}\normalfont
We give another proof of Lemma \ref{lemmab}.
To begin with, we introduce some notations.
For a cube free integer $A \in \bbz \setminus \{0, \pm1 \}$, we define
\[ s(A) = \dim_{\bbf_3}S(A) -1. \]
Let $m$ be the number of distinct prime factors $\equiv 2 \bmod 3$ of $A$.
Then we write
\[
s_0(A) = \begin{cases}
m & \textrm{if} \ A \equiv \pm3  \bmod 9, \\
m-1 & \textrm{if} \ A \equiv 0, \pm 2, \pm 4 \bmod 9, \\
m-2 & \textrm{if} \ A \equiv \pm 1 \bmod 9.
\end{cases}
\]

Heath-Brown pointed out in \cite[p. 247]{hb} that
\[ s(A) \equiv s_0(A) \bmod 2. \]
Indeed, 
Stephens \cite{steph} showed that $R(A) \equiv s(A) \bmod 2$ and determined the  root number, where $R(A)$ is the order of vanishing of the $L$-function $L(E_A / \bbq\ ; s)$ at the point $s=1$.
The precise formula is
\[ (-1)^{R(A)} = -w_3 \prod_{p \neq 3} w_p, \]
where
\begin{align*}
w_3 = -1\ &\textrm{if}\  A \equiv \pm1, \pm 3 \bmod 9, \ w_3 = 1 \ \textrm{otherwise}, \\
w_p = -1\ &\textrm{if}\ p|A, p \equiv 2  \bmod 3, \ w_p = 1 \ \textrm{otherwise}.
\end{align*}

In order to show Lemma \ref{lemmab}, we first check $p_1p_2^2 \in S(A)$ and $p_1 \not\in S(A)$ in the same way as the proof of Lemma \ref{lemmab}.
Since $S(A)$ is a subgroup of the group generated by $A, \zeta_3, p_1, p_2$, and contains $A$ and $p_1p_2^2$, we have $1 \leq s(A) \leq 3$.
On the other hand, we have $s_0(A) = 3-2 = 1$ and $p_1 \not\in S(A)$.
Therefore we can conclude that $s(A) = 1$ and $S(A) = \lan A, p_1p_2^2 \ran$.
\end{rem}

Now we are ready to prove our main theorem.
\begin{thm}\label{mainthm}
Let $p_1, p_2, p_3$ be rational primes such that $p_1,p_2, p_3 \equiv 2, 5 \bmod 9$, and 
\[ V : x_1^3 + p_1p_2x_2^3 + p_2p_3x_3^3 + p_3p_1x_4^3 = 0 \]
 a cubic surface over $\bbq$.
Assume that the Tate-Shafarevich group of any elliptic curve over $\bbq$ is finite.
Then $V(\bbq) \neq \emptyset$.
\end{thm}

\begin{proof}
We may assume that $p_1, p_2, p_3$ are distinct.
If $(p_1, p_2,p_3 ) \equiv (2,2,5)$ or $(5,5,2) \bmod 9$, then the curve 
\[ C_{p_1p_2p_3^2, p_1p_2^2} \cong \{ p_2p_3x^3 + p_3p_1 y^3 = z^3 \} \]
has a $k$-rational point by Lemma \ref{bf} and Lemma \ref{lemmaa}.
If $(p_1, p_2,p_3 ) \equiv (2,2,2)$ or $(5,5,5) \bmod 9$,
then the curve 
\[ C_{p_1^2p_2^2p_3^2, p_1p_2^2} \cong \{ p_2p_3x^3 + p_3p_1 y^3 = p_1p_2z^3 \} \]
has a $k$-rational point by Lemma \ref{bf} and \ref{lemmab}.
In either case the cubic surface $V$ has a $k$-rational point,
therefore $V$ has a $\bbq$-rational point.
\end{proof}

\begin{rem}\normalfont\label{rem1}
We do not need to assume the finiteness of the Tate-Shafarevich groups of all elliptic curves. Indeed, it is sufficient to assume that $\mathcyr{Sh}(E_{p_1p_2p_3^2}/\bbq)$ $($resp. $\mathcyr{Sh}(E_{p_1^2p_2^2p_3^2}/\bbq)$ $)$ is finite if $(p_1,p_2,p_3) \equiv (2,2,5)$ or $(5,5,2) \bmod 9$ (resp. $(2,2,2)$ or $(5,5,5) \bmod 9$).
\end{rem}

\section{Corollary of Swinnerton-Dyer's theorem}

Let $V$ be the smooth cubic surface over $\bbq$ defined by 
\[ a_1 x_1^3 + a_2 x_2^3 = a_3 x_3^3 + a_4x_4^3, \ \ a_1, a_2, a_3, a_4 \in \bbq^*.\]
Without loss of generality we can assume that $a_i$ are non-zero cube free integers and that for any prime $p$ the $4$-tuple $(v_p(a_1), v_p(a_2), v_p(a_3), v_p(a_4))$ is one of the following $4$-tuples up to permutation;
\[ (0,0,0,0), (0,0,0,1), (0,0,0,2), (0,0,1,1), (0,0,1,2).\] 
We denote by $V(\bba_\bbq)$ the set of adelic points of $V$.
Since $V$ is projective, we have
\[ V(\bba_\bbq) = \prod_v V(\bbq_v), \]
where $v$ runs over all places of $\bbq$.

The following theorem was shown by Swinnerton-Dyer.

\begin{thm}[{\cite[Theorem 1]{sd}}]\label{main}
Assume that the Tate-Shafarevich group of any elliptic curve defined by 
\[ x^3 + y^3 = Az^3 \]
over any quadratic field is finite.
If $V(\bba_\bbq) \neq \emptyset$, then each of the following three criteria is sufficient for $V(\bbq) \neq \emptyset$.
\begin{itemize}
\item There are rational primes $p_1 , p_3$ not dividing $3$ such that $p_1 | a_1, p_1 \nmid a_2 a_3 a_4$, $p_3 | a_3, p_3 \nmid a_1 a_2 a_4$.

\item There is a rational prime $p$ not dividing $3$ such that 
$p |a_1, p \nmid a_2a_3a_4$, and $a_2, a_3, a_4$ are not all in the same coset of $(\bbq_p^*)^3$.

\item There is a rational prime $p$ not dividing $3$ such that exactly two of the $a_i$ are divided by $p$, and $V$ is not birationally equivalent to a plane over $\bbq_p$.
\end{itemize}
\end{thm}

Note that $V(\bba_\bbq) \neq \emptyset$ is equivalent to the following condition.
\begin{cond}\label{cond2}
For every rational prime $p$ there exists $C_p$ in $\bbq_p^*$ such that each of the two curves
\begin{equation}
a_1x_1^3 + a_2x_2^3 = C_p x_0^3, \ \ a_3x_3^3 + a_4x_4^3 = C_p x_0^3 \label{eq1}
\end{equation}
has a $\bbq_p$-rational point.
\end{cond}
Indeed, if the set $V(\bbq_p)$ of $\bbq_p$-rational points on $V$ is not empty, then it is dense in $V$.

Theorem \ref{main} is deduced from the following theorem.
\begin{thm}[{\cite[Theorem 3]{sd}}]\label{key}
Assume that the Tate-Shafarevich group of any elliptic curve
\[x^3 + y^3 =Az^3 \]
over any quadratic field is finite.
If we can choose the $C_p$ in Condition \ref{cond2} so that for some primes $p_1, p_3$ $($which may be the same$)$ the curve
\begin{equation}
a_3^2 x_1^3 + a_4^2 x_2^3 = a_1a_2a_3a_4C_{p_1}x_0^3 \label{eq2}
\end{equation}
has no $\bbq_{p_1}$-rational points, and the curve 
\begin{equation}
a_1^2 x_3^3 + a_2^2 x_4^3 = a_1a_2a_3a_4C_{p_3}x_0^3 \label{eq3}
\end{equation}
has no $\bbq_{p_3}$-rational points,
then $V(\bbq) \neq \emptyset$.
\end{thm}

\begin{rem}\normalfont
In fact, Swinnerton-Dyer \cite{sd} proved the above theorems in the case where, instead of $\bbq$, the base field is any number field not containing the primitive cube roots of unity.
\end{rem}

We give an analog of Theorem \ref{main}, which deals with the prime $3$.

\begin{thm}\label{sum}
Assume that the Tate-Shafarevich group of any elliptic curve defined by 
\[ x^3 + y^3 = Az^3 \]
over any quadratic field is finite.
If $V(\bba_\bbq) \neq \emptyset$, then each of the following criteria is sufficient for $V(\bbq) \neq \emptyset$.
\begin{enumerate}
\item $3 | a_1$, $3 \nmid a_2a_3a_4$, and at least two of $a_2, a_3, a_4$ are in the same coset of $(\bbq_3^{*})^3$.
Moreover there is a rational prime $p$ not dividing $3$ such that $p | a_3$ and $p \nmid a_1a_2a_4$.

\item $3 \| a_1$, $3 \nmid a_2a_3a_4$, and $a_2, a_3, a_4$ are not all in the same coset of $(\bbq_3^*)^3$.
Moreover there is a rational prime $p$ not dividing $3$ such that $p | a_1$ and $p \nmid a_2a_3a_4$.

\item $3^2 \| a_1$, $3\nmid a_2 a_3 a_4$, and exactly two of $a_2, a_3, a_4$ are in the same coset of $(\bbq_3^*)^3$.

\item $3 \| a_1$, $3 \| a_3$, $3 \nmid a_2a_4$, and $V$ is not birationally equivalent to a plane over $\bbq_3$. 
Furthermore, $a_1/a_3 \in (\bbq_3^*)^3$.

\item $3^2 \| a_1$, $3 \| a_3$, $3 \nmid a_2a_4$, and $V$ is not birationally equivalent to a plane over $\bbq_3$.
Moreover there is a rational prime $p$ such that $p$ divides exactly one of $a_1, a_2$, and $p \nmid a_3a_4$.

\item $3 \nmid a_1a_2a_3a_4$ and $V$ is not birationally equivalent to a plane over $\bbq_3$.
Then we may assume that $a_1/a_2, a_2/a_3 \in (\bbq_3^*)^3$, $a_3/a_4 \not\in (\bbq_3^*)^3$.
Moreover there is a rational prime $p$ such that $p | a_1$ and $p \nmid a_2a_3a_4$.
\end{enumerate}
\end{thm}

\begin{proof}
What we have to do is to find $C_{p_1}$ and $C_{p_3}$ satisfying the condition of Theorem \ref{key}.

We first remark when the smooth cubic curve
\[ C : ax^3 +by^3 = cz^3 \]
over $\bbq_3$
has a $\bbq_3$-rational point.
The group $\bbq_3^*/(\bbq_3^*)^3$ is generated by the images of $3$ and $2$.
Thus $C$ is isomorphic over $\bbq_3$ to one of the following curves
\[ 3^kx^3 + 2^iy^3 = 2^jz^3, \ \ x^3 + 3b'y^3 = 9c'z^3,\]
where $i,j,k \in \{ 0,1,2 \}$, and $b', c'$ are $3$-adic units.
The latter has no $\bbq_3$-rational points.
In the former case it has a $\bbq_3$-rational point if and only if
\[ k=0,  \{i,j \} \neq \{ 1,2 \}, \ \textrm{or}\  k= 1, \ \textrm{or}\ k=2, i=j.\]

i) 
Put $p_1 = p$ and $C_p = a_4$.
Then each of the two curves (\ref{eq1})
\[ a_1x_1^3 + a_2x_2^3 = a_4 x_0^3, \ \ a_3x_3^3 + a_4x_4^3 = a_4 x_0^3 \]
has a $\bbq_p$-rational point.
Indeed, since $p \nmid a_1a_2a_4$, the first curve does in the same way as Lemma \ref{3.2}.
The second curve has a point $(x_3,x_4,x_0 ) = (0,1,1)$. 
The curve (\ref{eq2}) $a_3^2 x_1^3 + a_4^2 x_2^3 = a_1a_2a_3a_4^2x_0^3$
has no $\bbq_p$-rational points.
This follows from $p | a_3$ and the same argument as in the proof of Lemma \ref{3.1}
 
Put $p_3 = 3$ and $C_3 = a_2$.
Then each of the two curves $($\ref{eq1}$)$
\[ a_1x_1^3 + a_2x_2^3 = a_2 x_0^3, \ \ a_3x_3^3 + a_4x_4^3 = a_2 x_0^3 \]
has a $\bbq_3$-rational point.
The first curve has the point $(x_1, x_2, x_0) = (0,1,1)$.
The second curve has a $\bbq_3$-rational point since at least two of $a_2, a_3, a_4$ are in the same coset of $(\bbq_3^*)^3$.
Moreover, it follows from $3 | a_1$ and $3 \nmid a_2a_3a_4$ that the curve (\ref{eq3}) $ a_1^2 x_3^3 + a_2^2 x_4^3 = a_1a_2^2a_3a_4x_0^3$
has no $\bbq_3$-rational points.

ii)
By permuting indexes if necessary, we may assume that $a_3/a_4 \not \in (\bbq_3^*)^3$.
Put $p_1 = 3$ and $C_3 = a_1$.
Then each of the two curves (\ref{eq1})
\[ a_1x_1^3 + a_2x_2^3 = a_1 x_0^3, \ \ a_3x_3^3 + a_4x_4^3 = a_1 x_0^3 \]
has a $\bbq_3$-rational point since $3 \| a_1$ and $3 \nmid a_3a_4$.
The curve (\ref{eq2}) $a_3^2 x_1^3 + a_4^2 x_2^3 = a_1^2a_2a_3a_4x_0^3$ does not have a $\bbq_3$-rational point.
Indeed, if it does, then $a_3/a_4 \in (\bbq_3^*)^3$.
This is a contradiction.

Put $v_3 = p$ and $C_p = a_2$.
Then each of the two curves $($\ref{eq1}$)$
\[ a_1x_1^3 + a_2x_2^3 = a_2 x_0^3, \ \ a_3x_3^3 + a_4x_4^3 = a_2 x_0^3 \]
has a $\bbq_p$-rational point since $p \nmid a_2a_3a_4$.
The curve (\ref{eq3}) $ a_1^2 x_3^3 + a_2^2 x_4^3 = a_1a_2^2a_3a_4x_0^3$
has no $\bbq_p$-rational points since $p |a_1$ and $p \nmid a_2a_3a_4$.

iii)
By permuting the indexes if necessary, we may assume that $a_2/a_3 \in (\bbq_3^*)^3$ and $a_3/a_4 \not\in (\bbq_3^*)^3$.
Put $p_1 = p_3 = 3$ and $C_3 = a_2$.
Each of the two curves (\ref{eq1})
\[ a_1x_1^3 + a_2x_2^3 = a_2x_0^3, \ \ a_3x_3^3 + a_4x_4^3 = a_2 x_0^3 \]
has a $\bbq_3$-rational point.
That is obvious for the first curve.
Since $a_2/a_3 \in (\bbq_3^*)^3$, the second curve is isomorphic to $x_3^3 + (a_4/a_3)x_4^4 = x_0^3$ over $\bbq_3$.
This has the point $(x_3, x_4, x_0 )= (1,0,1)$.
The curves (\ref{eq2}), (\ref{eq3}) 
\[ a_3^2 x_1^3 + a_4^2 x_2^3 = a_1a_2^2a_3a_4x_0^3,\ \ a_1^2 x_3^3 + a_2^2 x_4^3 = a_1a_2^2a_3a_4x_0^3 \]
have no $\bbq_3$-rational points, because $3^2 \| a_1$, $3 \nmid a_2a_3a_4$, and $a_3/a_4 \not\in (\bbq_3^*)^3$.

iv)
We recall that $V$ is birationally equivalent to a plane over $\bbq_3$ if and only if one of the ratios $a_1a_2/a_3a_4, a_1a_3/a_2a_4, a_1a_4/a_2a_3$ is in $(\bbq_3^*)^3$ \cite[Lemme 8]{ct}.
Therefore in this case we have $a_2/a_4 \not \in (\bbq_3^*)^3$.
Put $p_1 = p_3 = 3$ and $C_3 = a_1$, then both of the two curves $($\ref{eq1}$)$
\[ a_1x_1^3 + a_2x_2^3 = a_1 x_0^3, \ \ a_3x_3^3 + a_4x_4^3 = a_1 x_0^3 \]
have a $\bbq_3$-rational point since $a_1/a_3 \in (\bbq_3^*)^3$.
Neither of the curves (\ref{eq2}) and (\ref{eq3})
\[ a_3^2 x_1^3 + a_4^2 x_2^3 = a_1^2a_2a_3a_4x_0^3, \ \ a_1^2 x_3^3 + a_2^2 x_4^3 = a_1^2a_2a_3a_4x_0^3 \]
has a $\bbq_3$-rational point.
Indeed, since $3 \| a_1$, $3\| a_3$, and $a_1/a_3 \in (\bbq_3^*)^3$, each of the curves (\ref{eq2}) and (\ref{eq3}) has a $\bbq_3$-rational point if and only if $a_2/a_4 \in (\bbq_3^*)^3$.

v)
Put $p_1 =3$ and $C_3 = a_2$.
Then the two curves (\ref{eq1})
\[a_1x_1^3 + a_2x_2^3 = a_2x_0^3, \ \ a_3x_3^3 + a_4x_4^3 = a_2x_0^3 \]
have a $\bbq_3$-rational point since $3 \| a_3, 3\nmid a_2a_4$.
The curve (\ref{eq2}) $a_3^2 x_1^3 + a_4^2 x_2^3 = a_1a_2^2a_3a_4x_0^3 $ has no $\bbq_3$-rational points.
To see this, note that the curve (\ref{eq2}) has a $\bbq_3$-rational point if and only if $a_1a_2^2a_3a_4/a_4^2 \in (\bbq_3^*)^3$, and this is equivalent to $a_1a_3/a_2a_4 \in (\bbq_3^*)^3$.
However, this contradicts the assumption that $V$ is birationally equivalent to a plane over $\bbq_3$ by \cite[Lemme 8]{ct}.

Put $p_3 = p$ and 
\[
C_p = 
\begin{cases}
a_2 & \textrm{if} \ p | a_1, \\
a_1 & \textrm{if} \ p | a_2.
\end{cases}
\]
Then the two curves $($\ref{eq1}$)$
\[ a_1x_1^3 + a_2x_2^3 = C_p x_0^3, \ \ a_3x_3^3 + a_4x_4^3 = C_p x_0^3 \]
have a $\bbq_p$-rational point.
The first curve does since $C_p$ is equal to $a_1$ or $a_2$.
The second also does, because $p \nmid a_3 a_4 C_p$.
Moreover the curve (\ref{eq3})
\[a_1^2 x_3^3 + a_2^2 x_4^3 = a_1a_2a_3a_4C_px_0^3 \]
has no $\bbq_p$-rational points 
since the triple $(v_p(a_1^2), v_p(a_2^2),v_p(a_1a_2a_3a_4C_p))$ is congruent to $(0,1,2)$ modulo $3$ up to permutation.

vi)We may assume that the $4$-tuple $(a_1, a_2, a_3, a_4)$ is congruent to
\[ (2^i, 2^i, 2^j, 2^k) \]
modulo $(\bbq_3^*)^3$, where $(i,j,k) = (0,0,0), (0,0,1), (0,0,2), (0,1,1), (0,1,2)$.
Since $V$ is not birationally equivalent to a plane over $\bbq_3$,
we have $(i,j,k) = (0,0,1), (0,0,2)$ by \cite[Lemme 8]{ct}.
Put $p_1 = 3$ and $C_3 = a_1$.
Then the two curves (\ref{eq1})
\[ a_1x_1^3 + a_2x_2^3 =  a_1x_0^3, \ \ a_3x_3^3 + a_4x_4^3 = a_1 x_0^3 \]
have a $\bbq_3$-rational point since $a_1/a_3 \in (\bbq_3^*)^3$.
The curve (\ref{eq2}) $a_3^2 x_1^3 + a_4^2 x_2^3 = a_1^2a_2a_3a_4x_0^3$
has no $\bbq_3$-rational points since $(a_3^2, a_4^2, a_1^2a_2a_3a_4)$ is congruent to $(2^0, 2^1,2^2)$ modulo $(\bbq_3^*)^3$ up to permutation.

Put $p_3 = p$ and $C_p = a_2$.
Each of the two curves (\ref{eq1})
\[ a_1x_1^3 + a_2x_2^3 = a_2 x_0^3, \ \ a_3x_3^3 + a_4x_4^3 = a_2 x_0^3 \]
has a $\bbq_p$-rational point since $p \nmid a_2a_3a_4$.
Moreover the curve (\ref{eq3}) $ a_1^2 x_3^3 + a_2^2 x_4^3 = a_1a_2^2a_3a_4x_0^3$
has no $\bbq_p$-rational points since $p | a_1$.
This completes the proof.
\end{proof}

\begin{rem}\normalfont\label{rem34}
We cannot apply Theorem \ref{main} or Theorem \ref{key} (hence Theorem \ref{sum}) to Theorem \ref{mainthm}.
Indeed, in Theorem \ref{mainthm}, we can easily see that there does not exist the $C_p$ satisfying the condition of Theorem \ref{key}.
\end{rem}

\begin{ack*}\normalfont
I would like to thank my supervisor, Professor Masaki Hanamura, for helpful comments and warm encouragement.
This work was supported by JSPS KAKENHI Grant Number $25 \cdot 1135$.
\end{ack*}

\end{document}